\documentclass[12pt,reqno] {amsart}%

\allowdisplaybreaks[3]
\usepackage{amsmath,amsthm,amsfonts,amscd,amssymb,latexsym}
\usepackage{mathrsfs,mathtools,verbatim}
\usepackage{bbm}
\usepackage{graphicx}
\usepackage{indentfirst}
\usepackage{enumerate}
\usepackage{hyperref}
\usepackage{cleveref}
\usepackage{xcolor}
\numberwithin{equation}{section}

\newtheorem{theorem}{Theorem}[section]

\newtheorem{definition}[theorem]{Definition}

\newtheorem{lemma}[theorem]{Lemma}

\newtheorem{proposition}[theorem]{Proposition}
\newtheorem{remark}[theorem]{Remark}

\newtheorem{question}[theorem]{Question}

\begin{document}
\title[Hochschild cohomology for free semigroup algebras]{Hochschild cohomology for free semigroup algebras}

\author{Linzhe Huang}
\address{Linzhe Huang, Beijing Institute of Mathematical Sciences and Applications, Beijing, 101408, China}
\email{huanglinzhe@bimsa.cn}
\author{Minghui Ma}
\address{Minghui Ma, School of Mathematical Sciences, Dalian University of Technology, Dalian, 116024, China}
\email{minghuima@dlut.edu.cn}
\author{Xiaomin Wei}
\address{Xiaomin Wei, School of Physical and Mathematical Sciences, Nanjing Tech University, Nanjing, 211816, China}
\email{wxiaomin@njtech.edu.cn}
\maketitle

\begin{abstract}
This paper focuses on the cohomology of operator algebras associated with the free semigroup generated by the set $\{z_{\alpha}\}_{\alpha\in\Lambda}$, with the left regular free semigroup algebra $\mathfrak{L}_{\Lambda}$ and the non-commutative disc algebra $\mathfrak{A}_{\Lambda}$ serving as two typical examples.
We establish that all derivations of these algebras are automatically continuous.
By introducing a novel computational approach, we demonstrate that the first Hochschild cohomology group of $\mathfrak{A}_{\Lambda}$ with coefficients in  $\mathfrak{L}_{\Lambda}$ is zero.
Utilizing the Ces\`aro operators and conditional expectations, we show that the first normal cohomology group of $\mathfrak{L}_{\Lambda}$ is trivial.
Finally, we prove that the higher cohomology groups of the non-commutative disc algebras with coefficients in the complex field vanish when $|\Lambda|<\infty$.
These methods extend to compute the cohomology groups of a specific class of operator algebras generated by the left regular representations of cancellative semigroups, which notably include Thompson's semigroup.
\end{abstract}



\section{Introduction}

In their paper \cite{DP99}, Davidson and Pitts introduced a novel category of intriguing operator algebras, termed as {\em free semigroup algebras}, and examined their invariant subspace structure.
Recently, many significant progresses have been made in the study of free semigroup algebras.
In \cite{DKP01}, Davidson, Katsoulis and Pitts  presented a comprehensive structure theorem applicable to all free semigroup algebras.
Later in \cite{DLP05}, Davidson, Li and Pitts  established a Kaplansky density theorem for free semigroup algebras.
These algebras
 are weak-operator closed and generated by a (countable or uncountable) set of isometries $\{L_{\alpha}\}_{\alpha\in\Lambda}$  with pairwise orthogonal ranges, acting on a Hilbert space.
These conditions can be expressed algebraically as follows:
\begin{align*}
    L_{\alpha}^*L_{\alpha}=I\quad\text{for}~ \alpha\in\Lambda\quad\quad\text{and}\quad \quad \sum_{\alpha\in\Lambda}L_{\alpha}L_{\alpha}^*\leqslant I,
\end{align*}
where $\sum_{\alpha\in\Lambda}$ denotes the sum in the strong-operator topology.
A typical example of free semigroup algebras is the non-self-adjoint algebra $\mathfrak{L}_{\Lambda}$ generated by the left regular representation of  $\mathbb{F}^+_{\Lambda}$ acting on the Hilbert space $\ell^2(\mathbb{F}^+_{\Lambda})$.
This representation is a commonly used method to construct operator algebras, such as the hyperfinite $\mathrm{II}_1$ factor and free group factors.
The algebra $\mathfrak{L}_{\Lambda}$ is referred to as the {\em left regular free semigroup algebra}, or alternatively as the {\em non-commutative analytic Toeplitz algebra}. 
The study of this algebra is initiated by  Popescu \cite{Pop89a,Pop89b,Pop95}.
Up to now, numerous studies have been conducted on this algebra, and we refer to existing literatures e.g. \cite{AP95,AP00,DP98a,DP98b,DP99,Ken13} for further exploration.
The unital norm closed algebra $\mathfrak{A}_{\Lambda}$ generated by the set of isometries $\{L_{\alpha}\}_{\alpha\in\Lambda}$ with pairwise orthogonal ranges is referred to as the {\em non-commutative disc algebra} by Popescu \cite{Pop91}.
He computed the first cohomology group of this Banach algebra with coefficients in the complex field, demonstrating its non-amenability \cite{Pop96}.
Clou\^atre, Martin and Timko \cite{CMT22} investigated analytic functionals for the non-commutative disc algebra.

The Hochschild cohomology theory of operator algebras has a history spanning over half a century, see e.g. \cite{CPSS03,Joh72,Kad66,PS10,Sak66} for the development.
The cohomology groups serve as isomorphism invariants, aiding in the classification of operator algebras.
Particularly, the first cohomology groups characterize the structures of derivations, which is a vast and intricate subject that experienced significant development during the 1960s and 1970s, and has a close relation with quantum physics \cite{KL14}.

Research on cohomology in von Neumann algebras is extensive.
In a series of papers \cite{J,JKR,KR,KR2}, Johnson, Kadison and Ringrose proved that the cohomology groups $\mathscr{H}^{n}(\mathcal{M},\mathcal{M})$ are zero for all $n\geqslant 1$ when $\mathcal{M}$ is a hyperfinite von Neumann algebra.
In \cite{SS}, Sinclair and Smith showed that $\mathscr{H}^{n}(\mathcal{M},\mathcal{M})=0$ for von Neumann algebras with Cartan subalgebras and separable preduals.
Christensen et al. \cite{CPSS03} proved that $\mathscr{H}^n(\mathcal{M},\mathcal{M})=0$ and $\mathscr{H}^n(\mathcal{M}, B(\mathcal{H}))=0$ for type $\mathrm{II}_1$ factors with property $\Gamma$.
In \cite{PS10}, Pop and Smith proved that $\mathscr{H}^{2}(\mathcal{M},\mathcal{M})=0$ for type $\mathrm{II}_1$ von Neumann algebras that are not prime.
A remaining significant issue pertains to the higher cohomology groups of free group factors.
In computations of the cohomology groups of von Neumann algebras, the normal cohomology is a significant tool as $\mathscr{H}_w^n(\mathcal{M},\mathcal{M})=\mathscr{H}^n(\mathcal{M},\mathcal{M})$, see e.g. \cite{SS95}.
Studies on cohomology in Banach algebras are deeply related to the amenability of groups.
Johnson \cite{Joh72} proved that a locally compact group $G$ is amenable if and only if $\mathscr{H}^1(\ell^1(G),\mathcal{X}^*)=0$ for all Banach $\ell^1(G)$-bimodule $\mathcal{X}$.
Another topic of cohomology theory is the automatic continuity of derivations.
Kaplansky \cite{Kap53} conjectured that all derivations of a $C^*$-algebra are automatically continuous and Sakai \cite{Sak60} solved this problem.
Kadison \cite{Kad66} proved that all derivations of a $C^*$-algebra are normal.
Johnson and Sinclair \cite{JS68} showed that all derivations of a semisimple Banach algebra are continuous.

In this paper, instead of studying cohomology of abstract Banach algebras, we focus on two typical objects: the left regular free semigroup algebra $\mathfrak{L}_{\Lambda}$ and the non-commutative disc algebra $\mathfrak{A}_{\Lambda}$.
We define a total order on the free semigroup $\mathbb{F}_{\Lambda}^+$, which is well-ordered and multiplication-invariant.
The order is interesting in itself and can be used to prove the semisimplicity of these algebras, see also \cite[Theorem 1.7]{DP99}.
Thus we can immediately obtain that all derivations of these algebras are continuous by \cite[Theorem 4.1]{JS68}.

We develop an original method to compute the first continuous cohomology group of $\mathfrak{A}_{\Lambda}$ with coefficients in $\mathfrak{L}_{\Lambda}$.
Let $\{L_{\alpha}\}_{\alpha\in\Lambda}$ be the generators of $\mathfrak{A}_{\Lambda}$ (see \S \ref{sec:Banach algebras associated with free semigroup}).
Then every continuous derivation $D$ from $\mathfrak{A}_{\Lambda}$ into $\mathfrak{L}_{\Lambda}$ is determined by $\{D(L_{\alpha})\}_{\alpha\in\Lambda}$.
We prove that $D$ is an inner derivation by the following three steps:
\begin{enumerate}[(i)]
\item  By extending the averaging method used for computing cohomology groups of von Neumann algebras to free semigroup algebras, we show that $D(L_w)=S_w-L_w^*S_wL_w$ for some $S_w\in\mathfrak{L}_{\Lambda}$, $\forall w\in\mathbb{F}_{\Lambda}^+$;

\item The second step is the difference between the computations of cohomology groups of self-adjoint algebras and non-self-adjoint algebras as $L_w^*$ does not belong to $\mathfrak{L}_{\Lambda}$.
    We do further efforts to show that $S_w$ can be replaced by $L_wT_w$ for some $T_w\in \mathfrak{L}_{\Lambda}$.
    Hence  $D$ is {\em locally inner} at $L_w$, i.e., $D(L_w)=L_wT_w-T_wL_w$;

\item Suppose that $D(L_{\alpha})=0$.
    For any $\beta\in\Lambda\setminus\{\alpha\}$, we prove that $D(L_{\beta})=L_{\beta}T_{\beta}-T_{\beta}L_{\beta}$ for some $T_{\beta}\in \mathfrak{L}_{\Lambda}$ with $[T_{\beta},L_{\alpha}]=0$.
    Furthermore, assuming $D(L_{\alpha})=D(L_{\beta})=0$, we establish that $D(L_{\gamma})=0$ for any $\gamma\in\Lambda\setminus\{\alpha,\beta\}$.
\end{enumerate}
In step (iii), there is no limit to the cardinality of the index set $\Lambda$.
Therefore, we allow the number of generators to be an uncountable set.
The above ideas and methods have been applied to compute the cohomology groups of the operator algebra associated with Thompson's semigroup in a recent companion paper \cite{Hua21}.
This part of work is inspired by a question proposed in \cite[Section 6]{Hua21}.
The free semigroup and Thompson's semigroup (see e.g. \cite{GMQ,Xue21}) are both cancellative semigroups.
We believe that our techniques are applicable to investigate the cohomology of a wide class of operator algebras generated by left regular representations of amenable or non-amenable cancellative semigroups.
We also hope that our work could bring inspiration to the cohomology problem of free group factors.

The left regular free semigroup algebra adopts the weak-operator topology induced by $B(\ell^2(\mathbb{F}_{\Lambda}^+))$.
We define the normal cohomology of this weak-operator closed algebra, which is an analogue of the concept of normal cohomology in von Neumann algebras.
Applying the Ces\`{a}ro operators (see \cite[Lemma 1.1]{DP99}), we prove that the first normal cohomology group $\mathscr{H}_{w}^1(\mathfrak{L}_{\Lambda},\mathfrak{L}_{\Lambda})$ is zero.

Popescu proved  $\mathscr{H}^1(\mathfrak{A}_{\Lambda},\mathbb{C})\cong (\mathbb{C}^{|\Lambda|},+)$ when $|\Lambda|<\infty$ and $\mathscr{H}^1(\mathfrak{A}_{\Lambda},\mathbb{C})\cong (\ell^2(\Lambda),+)$ when $|\Lambda|=\infty$.
Inspired by his work, we study the higher cohomology groups  $\mathscr{H}^n(\mathfrak{A}_{\Lambda},\mathbb{C})$ for $n\geqslant2$.
Given an $n$-cocycle $\phi$, we construct an $(n-1)$-linear map $\psi$ such that $\phi=\partial^{n-1}\psi$ by a cutting operation on $\mathbb{F}_{\Lambda}^+$.
In a word, we prove that $\mathscr{H}^n(\mathfrak{A}_{\Lambda},\mathbb{C})=0$ for $n\geqslant2$ when $|\Lambda|<\infty$.
The finiteness of $\Lambda$ assures the boundedness of the $(n-1)$-linear map $\psi$.

The paper is organized as follows: In \S \ref{sec:free semigroup}, we introduce some  basic properties of the free semigroup and  the left regular free semigroup algebra.
In \S \ref{sec:continuity derivation}, we establish that all derivations of the left regular free semigroup algebra and the non-commutative disc algebra are automatically continuous.
In \S \ref{sec:first cohomology}, we prove that the first cohomology of $\mathfrak{A}_{\Lambda}$ into $\mathfrak{L}_{\Lambda}$ vanishes.
In \S \ref{sec:normal cohomology}, we show that the first normal cohomology group of $\mathfrak{L}_{\Lambda}$ is trivial.
In \S \ref{sec:higher order cohomology}, we prove that the higher cohomology group of $\mathfrak{A}_{\Lambda}$ with coefficients in complex field are trivial when $|\Lambda|<\infty$.

\textbf{Acknowledgements.} We are deeply grateful to Professor Liming Ge for his continuous encouragement and invaluable discussions.

\section{Operator algebras associated with free semigroups}\label{sec:free semigroup}
In this section, we introduce some fundamental properties of free semigroups and free semigroup algebras, which will be utilized in the computation of cohomology.

\subsection{The free semigroups}
Let $\mathbb{F}^+_{\Lambda}$ be the free semigroup generated by $\{z_{\alpha}\}_{\alpha\in\Lambda}$, where the non-empty index set $\Lambda$ can be countable or uncountable.
A word in $\mathbb{F}^+_{\Lambda}$ is a juxtaposition of the symbols $\{z_{\alpha}\}_{\alpha\in\Lambda}$ in any order with any number of repetitions.
The {\em length} of a word $w=z_{\alpha_1}\cdots z_{\alpha_n}$ is defined to be $|w|=n$.
The length of the unit $e$ (the empty word) in $\mathbb{F}^+_{\Lambda}$ is zero.
For any words $w,u\in\mathbb{F}^+_{\Lambda}$, we say $w$ is {\em left divided} by $u$, denoted by $u|w$, if there exists a word $v\in\mathbb{F}^+_{\Lambda}$ such that $w=uv$; otherwise we write $u\nmid w$.
We say $w$ is {\em right divided} by $u$, written as $u\ddag w$, if there exists a word $v\in\mathbb{F}^+_{\Lambda}$ such that $w=vu$.
These two relations are partial orders on $\mathbb{F}^+_{\Lambda}$.
We refer to \cite{Xue21} for an interesting non-commutative arithmetic theory of cancellative semigroups, where the division relations are well investigated.

The free semigroup $\mathbb{F}^+_{\Lambda}$ is highly non-commutative.
One can check by induction on the length of words that for any $w,u\in\mathbb{F}^+_{\Lambda}$, if $wu=uw$, then $w=v^m$ and $u=v^n$ for some $v \in\mathbb{F}^+_{\Lambda}$ and integers $m,n\geqslant 0$.
Now we present the following technical and useful lemma.

\begin{lemma}\label{lem:commutative}
Let $w,u,v\in\mathbb{F}^+_{\Lambda}$ with $w\ne e$.
If there is a positive integer $k\geqslant \frac{|u|}{|w|}  +1$ such that $vw^{k}=w^{k}u$, then $uw=wu$ and $u=v$.
\end{lemma}

\begin{proof}
The case $u=e$ is trivial.
We assume that $u\ne e$.
Let $k_0=\lceil\frac{|u|}{|w|}\rceil\footnote{$\lceil x\rceil$ is the ceiling function maps $x$ to the smallest integer greater than or equal to $x$.} \geqslant 1$.
Then $k\geqslant k_0+1$ and $k_0|w|\geqslant|u|$.
It follows from $vw^{k-k_0}w^{k_0}=vw^k=w^ku$ that $w^{k_0}=u_1u$ for some $u_1\in\mathbb{F}^+_{\Lambda}$.
Since
\begin{equation*}
  |u_1|+|u|=k_0|w|<\bigg(\frac{|u|}{|w|}+1\bigg)|w|=|u|+|w|,
\end{equation*}
i.e., $|u_1|<|w|$, we can write $w=u_1u_2$ for some $u_2\in\mathbb{F}^+_{\Lambda}$.
Then $u=u_2w^{k_0-1}$.
Since
\begin{equation*}
  vw^{k-k_0-1}u_1u_2u_1u_2w^{k_0-1}=vw^k=w^ku=w^{k-1}u_1u_2u_2w^{k_0-1},
\end{equation*}
by the cancellation law, we can get $vw^{k-k_0-1}u_1u_2u_1=w^{k-1}u_1u_2$.
Therefore, $u_2u_1=u_1u_2$.
Thus, $w=u_1u_2$ commutes with $u=u_2(u_1u_2)^{k_0-1}$ and hence $u=v$.
\end{proof}

\subsection{The free semigroup algebras}\label{sec:Banach algebras associated with free semigroup}
Let $\mathcal{H}_{\Lambda}$ be the Hilbert space $\ell^2(\mathbb{F}^+_{\Lambda})$ consisting of all square summable functions on the free semigroup $\mathbb{F}^+_{\Lambda}$, which serves as the natural Hilbert space for the left regular representation.
The family of functions $\{\xi_w\colon w\in\mathbb{F}^+_{\Lambda}\}$ forms an orthonormal basis for $\mathcal{H}_{\Lambda}$, where $\xi_w$ is the indicator function taking value $1$ at the point $w$ and $0$ elsewhere.
The generators of the left regular representation, denoted by $L_{\alpha}$, act as $L_{\alpha}\xi_{w}=\xi_{z_{\alpha}w}$ for $w\in\mathbb{F}_{\Lambda}^+$.
They are isometries with pairwise orthogonal ranges acting on $\mathcal{H}_{\Lambda}$.
More precisely, we have
\begin{align*}
  L_{\alpha}^*L_{\beta}=\delta_{\alpha,\beta}I\quad\text{for}~ \alpha,\beta\in\Lambda\quad\quad\text{and}\quad \quad \sum_{\alpha\in\Lambda}L_{\alpha}L_{\alpha}^*=I-P_e,
\end{align*}
where $P_e$ is the rank one projection onto $\mathbb{C}\xi_e$.
The weak-operator closed  algebra $\mathfrak{L}_{\Lambda}$ generated by $\{L_{\alpha}\}_{\alpha\in\Lambda}$ is called the {\em left regular free semigroup algebra}, and the norm closed algebra $\mathfrak{A}_{\Lambda}$ generated by $\{L_{\alpha}\}_{\alpha\in\Lambda}$ is called the {\em non-commutative disc algebra}.
The non-commutative disc algebra is a proper Banach subalgebra of the left regular free semigroup algebra.

The generators $\{L_{\alpha}\}_{\alpha\in\Lambda}$ induce the action of all words.
For $u=z_{\alpha_1}z_{\alpha_2}\cdots z_{\alpha_n}\in\mathbb{F}^+_{\Lambda}$, let
\begin{align*}
    L_u=L_{\alpha_1}L_{\alpha_2}\cdots L_{\alpha_n},
\end{align*}
which gives $L_u\xi_{w}=\xi_{uw}$.
The operator $L_u$ is an isometry on $\mathcal{H}_{\Lambda}$ with range $\ell^2(u\mathbb{F}^+_{\Lambda})$, where $u\mathbb{F}^+_{\Lambda}=\{w\in\mathbb{F}^+_{\Lambda}\colon u|w\}$.
Moreover, for $u,v\in\mathbb{F}^+_{\Lambda}$, we have
\begin{align*}
    L_u^*L_v=\begin{cases}
        L_{u^{-1}v}&u|v,\\
        L_{v^{-1}u}^*&v|u,\\
        0&\text{otherwise}.
    \end{cases}
\end{align*}
For any two functions $\varphi,\psi\in\mathcal{H}_{\Lambda}$, their {\em convolution} is defined by
\begin{equation*}
\varphi *\psi(w)=\sum\limits_{u\in \mathbb{F}^+_{\Lambda},u|w} \varphi(u)\psi(u^{-1}w),
\end{equation*}
Since the number of left divisors of $w$ is $|w|+1<\infty$, the above sum is a finite sum.
For each $\varphi\in\mathcal{H}_{\Lambda}$, the convolution induces an operator $L_{\varphi}$ on $\mathcal{H}_{\Lambda}$ given by
\begin{align*}
    L_{\varphi}\psi=\varphi *\psi,
\end{align*}
which is called the {\em left convolution operator} induced by $\varphi$.
Notice that $L_{\varphi}$ may be unbounded since $\varphi *\psi$ is not always square summable.
We shall note that the three notations $L_{\alpha},L_u,L_{\varphi}$ are consistent:
\begin{align*}
  L_{\alpha}=L_{z_{\alpha}}=L_{\xi_{z_{\alpha}}}\quad\text{and}\quad    L_u=L_{\xi_u}.
\end{align*}
We denote the set of all bounded left convolution operators by
\begin{equation*}
  \mathscr{L}_{\Lambda}=\{L_{\varphi}\in B(\mathcal{H}_{\Lambda})\colon\varphi \in \mathcal{H}_{\Lambda}\}.
\end{equation*}
For two left convolution operators $L_{\varphi},L_{\psi}\in \mathscr{L}_{\Lambda}$, their composition is $L_{\varphi}L_{\psi}=L_{\varphi* \psi}\in \mathscr{L}_{\Lambda}$.
Similarly, we define the {\em right convolution operator} induced by $\varphi$ by
\begin{align*}
    R_{\varphi}\psi=\psi*\varphi.
\end{align*}
The set of all bounded right convolution operators is denoted by
\begin{equation*}
  \mathscr{R}_{\Lambda}=\{R_{\varphi}\in B(\mathcal{H}_{\Lambda})\colon\varphi\in\mathcal{H}_{\Lambda}\}.
\end{equation*}
For two right convolution operators $R_{\varphi},R_{\psi}\in \mathscr{R}_{\Lambda}$, their composition is $R_{\varphi}R_{\psi}=R_{\psi* \varphi}\in \mathscr{R}_{\Lambda}$.

For any subset $\mathcal{A}$ of $B(\mathcal{H}_{\Lambda})$, its {\em commutant} is defined as
\begin{equation*}
  \mathcal{A}'=\{T\in B(\mathcal{H}_{\Lambda})\colon TA=AT,~\forall A\in\mathcal{A}\}.
\end{equation*}
We prove the following result following \cite[Theorem 6.7.2]{Kadison}.

\begin{proposition}\label{prop:commutant}
 The following relations hold:
 \begin{equation*}
\mathscr{L}_{\Lambda}'=\mathscr{R}_{\Lambda},\quad
\mathscr{R}_{\Lambda}'=\mathscr{L}_{\Lambda}.
\end{equation*}
\end{proposition}

\begin{proof}
It is clear that $\mathscr{L}_{\Lambda}\subseteq\mathscr{R}_{\Lambda}'$.
Suppose that $T\in\mathscr{R}_{\Lambda}'$.
For $v\in\mathbb{F}^+_{\Lambda}$, we consider the following two linear functionals on $\mathcal{H}_{\Lambda}$:
\begin{align*}
   \psi\mapsto(T\psi)(v)\quad\text{and}\quad \psi\mapsto(T\xi_e*\psi)(v).
\end{align*}
Clearly, the first linear functional is bounded.
The second one is also bounded since
\begin{align*}
    |(T\xi_e* \psi)(v)|\leqslant\sum_{v_1v_2=v}|T\xi_e(v_1)||\psi(v_2)|\leqslant \|T\xi_e\|_2\|\psi\|_2.
\end{align*}
For any $u\in \mathbb{F}^+_{\Lambda}$, we have
\begin{equation*}
  (T\xi_u)(v)=(TR_{u}\xi_{e})(v)=(R_{u}T\xi_{e})(v)=(T\xi_{e}*\xi_u)(v).
\end{equation*}
Since the two continuous linear functionals take the same values at each member of the orthonormal basis $\{\xi_{u}\}_{u\in\mathbb{F}^+_{\Lambda}}$, we have
\begin{equation*}
  (T\psi)(v)=(T\xi_e*\psi)(v)
\end{equation*}
for any $\psi\in\mathcal{H}_{\Lambda}$.
It follows that $T=L_{T\xi_e}\in\mathscr{L}_{\Lambda}$.
Therefore, $\mathscr{R}_{\Lambda}'=\mathscr{L}_{\Lambda}$, and symmetrically $\mathscr{L}_{\Lambda}'=\mathscr{R}_{\Lambda}$.
\end{proof}

\begin{remark}
Proposition \ref{prop:commutant} also appears in  \cite{Pop95} (see Proposition 1.1 and Theorem 1.2 of \cite{Pop95}).
\end{remark}

\begin{remark}
Let $\mathcal{S}$ be a cancellation semigroup.
\begin{enumerate}[$(i)$]
\item For any cancellative semigroup $\mathcal{S}$, let $\mathscr{L}(\mathcal{S})$ (resp.  $\mathscr{R}(\mathcal{S})$) be the algebra consisting of all bounded left (resp. right) convolution operators.
    Then $\mathscr{L}(\mathcal{S})$ and $\mathscr{R}(\mathcal{S})$ are the commutants of each other.

\item Let $\mathfrak{L}(\mathcal{S})$ be the weak-operator closed algebra generated by $\{L_u\colon u\in\mathcal{S}\}$, which is a subalgebra of $\mathscr{L}(\mathcal{S})$.
    It is well known that $\mathfrak{L}(\mathcal{S})=\mathscr{L}(\mathcal{S})$ when $\mathcal{S}$ is a group.
    But we do not know if they are equal for every cancellation semigroup $\mathcal{S}$.
    We would like to know whether this holds for Thompson's semigroup
\end{enumerate}
\end{remark}

From Proposition \ref{prop:commutant}, we know that $\mathscr{L}_{\Lambda}$ is a weak-operator closed algebra containing $\{L_{\alpha}\}_{\alpha\in\Lambda}$.
Hence,  $\mathscr{L}_{\Lambda}$ contains $\mathfrak{L}_{\Lambda}$.
We next show that  $\mathscr{L}_{\Lambda}=\mathfrak{L}_{\Lambda}$.

Let $\mathbb{F}^+_{\Lambda,k}=\{w\in\mathbb{F}^+_{\Lambda}\colon|w|=k\}$, the set of words in the free semigroup $\mathbb{F}^+_{\Lambda}$ with length $k$.
We denote by $Q_k$ the projection onto $\ell^2(\mathbb{F}^+_{\Lambda,k})$.
For any $j\in\mathbb{Z}$, let $\Phi_j$ be the linear operator on  $B(\mathcal{H}_{\Lambda})$ given by
\begin{align*}
    \Phi_j(T)=\sum_{k\geqslant \max\{ 0,j\}}Q_kTQ_{k-j}.
\end{align*}
Then $\Phi_j$ is a completely contractive projection.
Moreover, for $j\geqslant0$, we have
\begin{equation*}
  \Phi_j(L_{\varphi})=\sum\limits_{w\in \mathbb{F}^+_{\Lambda},\,|w|=j} \varphi(w)L_w\quad\text{for}~L_{\varphi}\in\mathscr{L}_{\Lambda}.
\end{equation*}
The following lemma appears in \cite[Lemma 1.1]{DP99}.

\begin{lemma}\label{lem: Cesaro operators}
The Ces\`aro operators on $B(\mathcal{H}_{\Lambda})$ defined by
\begin{equation*}
  \Sigma_k(T)=\sum_{|j|<k}\left(1-\frac{|j|}{k} \right) \Phi_j(T),\quad k\geqslant 1
\end{equation*}
are completely positive contractions.
Moreover, for any $T\in B(\mathcal{H}_{\Lambda})$, $\Sigma_k(T)$ converges to $T$ in the strong-operator topology.
\end{lemma}

It is clear that $\Sigma_k(L_\varphi)$ is a finite sum for $L_{\varphi}\in\mathscr{L}_{\Lambda}$ when $|\Lambda|<\infty$.
But this is not true for the case $|\Lambda|=\infty$.
We need to overcome this difficulty.

Suppose $\mathcal{I}$ is a finite subset of $\Lambda$.
Let $\mathbb{F}_{\mathcal{I}}^+$ be the free semigroup generated by $\{z_{\alpha}\}_{\alpha\in \mathcal{I}}$.
We define a linear map on $\mathscr{L}_{\Lambda}$ as follows:
\begin{equation*}
  \mathcal{E}_{\mathcal{I}}(L_{\varphi})
  =\sum_{w\in\mathbb{F}_{\mathcal{I}}^+}\varphi(w)L_w,\quad L_{\varphi}\in \mathscr{L}_{\Lambda}.
\end{equation*}
Due to the following lemma, $\mathcal{E}_{\mathcal{I}}$ is called the {\em conditional expectation} associated with $\mathcal{I}$.

\begin{lemma}\label{lem:conditional expectation}
The map $\mathcal{E}_{\mathcal{I}}$ is a completely contractive homomorphism.
Moreover, $\mathcal{E}_{\mathcal{I}}(L_{\varphi})$ converges to $L_{\varphi}$ in the strong-operator topology.
\end{lemma}

\begin{proof}
It is clear that $\mathcal{E}_{\mathcal{I}}$ is a homomorphism.
Let $X(\mathcal{I})=\{w\in\mathbb{F}_{\Lambda}^+\colon z_{\alpha}\nmid w,\forall \alpha\in\mathcal{I}\}$.
Then
\begin{equation*}
  \mathbb{F}_{\Lambda}^+=\bigsqcup_{w\in X(\mathcal{I})}\mathbb{F}_{\mathcal{I}}^+w,\quad
  \mathcal{H}_{\Lambda}=\bigoplus_{w\in X(\mathcal{I})}\ell^2(\mathbb{F}_{\mathcal{I}}^+w),
\end{equation*}
and
\begin{equation*}
  \mathcal{E}_{\mathcal{I}}(L_{\varphi})=\bigoplus_{w\in X(\mathcal{I})} R_wQ_{\mathcal{I}}L_{\varphi}Q_{\mathcal{I}}R_w^*,
\end{equation*}
where $Q_{\mathcal{I}}$ is the projection onto $\ell^2(\mathbb{F}_{\mathcal{I}}^+)$.
It follows that $\mathcal{E}_{\mathcal{I}}$ is completely contractive.
For any $w\in\mathbb{F}_{\Lambda}^+$, it is clear that
\begin{equation*}
  \lim_{\mathcal{I}}\|(\mathcal{E}_{\mathcal{I}}(L_{\varphi})
  -L_{\varphi})\xi_w\|_2=0.
\end{equation*}
Since the net $\{\mathcal{E}_{\mathcal{I}}(L_{\varphi})\}_{\mathcal{I}}$ is uniformly bounded, it converges to $L_{\varphi}$ in the strong-operator topology.
\end{proof}

\begin{proposition}\label{prop: two algebras same}
The algebra $\mathscr{L}_{\Lambda}$ is the weak-operator closure of $span\{L_w\colon w\in\mathbb{F}_{\Lambda}^+\}$.
That is $\mathscr{L}_{\Lambda}=\mathfrak{L}_{\Lambda}$.
\end{proposition}

\begin{proof}
Let $L_{\varphi}\in\mathfrak{L}_{\Lambda}$.
By \Cref{lem: Cesaro operators},
\begin{equation*}
  \Sigma_k(L_{\varphi})=\sum_{0\leqslant j<k}\left(1-\frac{j}{k} \right)\sum_{|w|=j}\varphi(w)L_w
\end{equation*}
converges to $L_{\varphi}$ in the strong-operator topology.
By \Cref{lem:conditional expectation}, $\mathcal{E}_{\mathcal{I}}( \Sigma_k(L_{\varphi}))$ converges to $\Sigma_k(L_{\varphi})$ in the strong-operator topology for every $k\geqslant 1$.
Note that $\mathcal{E}_{\mathcal{I}}(\Sigma_k(L_{\varphi}))$ is a finite sum since $\mathcal{I}$ is a finite set, i.e.,
\begin{equation*}
  \mathcal{E}_{\mathcal{I}}(\Sigma_k(L_{\varphi}))\in span\{L_w\colon w\in\mathbb{F}_{\Lambda}^+\}.
\end{equation*}
This completes the proof.
\end{proof}

\begin{remark}
It was shown that $\mathscr{L}_{\Lambda}=\mathfrak{L}_{\Lambda}$ when $|\Lambda|=n$ (see \cite[Theorem 1.2]{DP99}).
\end{remark}

In the end of this section, we show that $\mathfrak{L}_{\Lambda}$ is invariant under conjugation.

\begin{lemma}\label{lem:conjugation}
Let $w\in\mathbb{F}^+_{\Lambda}$.
Then $L_w^*\mathfrak{L}_{\Lambda}L_w\subseteq\mathfrak{L}_{\Lambda}$.
\end{lemma}

\begin{proof}
For any
\begin{equation*}
  L_{\varphi}=\sum\limits_{u\in\mathbb{F}^+_{\Lambda}}\varphi(u)L_u
  \in\mathfrak{L}_{\Lambda},
\end{equation*}
we need to prove that $L_w^*L_{\varphi}L_w\in\mathfrak{L}_{\Lambda}$.
We first prove that for every $u\in\mathbb{F}^+_{\Lambda}$, $L_w^*L_uL_w \in \mathfrak{L}_{\Lambda}$.
Without loss of generality, we assume that $L_w^*L_uL_w\ne 0$.
Then there are words $w_1, w_2\in \mathbb{F}^+_{\Lambda}$ such that
\begin{equation*}
  \langle L_w^*L_uL_w\xi_{w_1},\xi_{w_2}\rangle\ne 0,
\end{equation*}
which imples that $uww_1=ww_2$.
It follows from $|uw|\geqslant |w|$ that $uw=wv$ for some $v\in \mathbb{F}^+_{\Lambda}$.
Hence $L_w^*L_uL_w=L_w^*L_wL_v=L_v \in \mathfrak{L}_{\Lambda}$.
Thus, we have
\begin{align*}
  L_w^*L_{\varphi}L_w
  &=L_w^*\Big(\sum\limits_{u\in\mathbb{F}^+_{\Lambda}} \varphi(u)L_u\Big)L_w\\
  &=\sum\limits_{u,v\in \mathbb{F}^+_{\Lambda},\,uw=wv}\varphi(u)L_v
  =\sum\limits_{v\in\mathbb{F}^+_{\Lambda}} \varphi(wvw^{-1})L_v,
\end{align*}
where $\varphi(wvw^{-1})=0$ if $wvw^{-1} \notin \mathbb{F}^+_{\Lambda}$.
Hence $L_w^*L_{\varphi}L_w\in \mathfrak{L}_{\Lambda}$.
\end{proof}

\section{Automatic continuity of derivations}\label{sec:continuity derivation}
Let $\mathcal{B}$ be a unital Banach algebra over $\mathbb{C}$, and let $\mathcal{X}$ be a Banach $\mathcal{B}$-bimodule.
A {\em derivation} of $\mathcal{B}$ into $\mathcal{X}$ is a linear map $D$ satisfying the Leibniz rule: $D(AB)=AD(B)+D(A)B$ for all $A$,
$B\in\mathcal{B}$.
When considering the bimodule $\mathcal{X}$ as $\mathcal{B}$ itself, we simply call $D$ a derivation of $\mathcal{B}$.
We say $D$ is {\em continuous} if $\|D\|=\sup_{A\in\mathcal{B},\,\|A\|=1}\|D(A)\|<\infty$.
We will show that all derivations of $\mathfrak{L}_{\Lambda}$ and $\mathfrak{A}_{\Lambda}$ are continuous.

Let us  recall some significant definitions in set theory.
A {\em well-order} on a set $\mathscr{S}$ is a total ordering with the property that every non-empty subset of $\mathscr{S}$ has a least element in this ordering.
The set $\mathscr{S}$ together with the ordering is called a {\em well-ordered set}.

\begin{lemma}[The well-ordering theorem]
Every set can be well-ordered.
\end{lemma}

We assume that $(\Lambda,\preccurlyeq)$ is a well-ordered set by the well-ordering theorem.
For any $\alpha,\beta\in\Lambda$, we write $\alpha\prec\beta$ if $\alpha\preccurlyeq\beta$ and $\alpha\ne\beta$.
In the following, we introduce an order relation on the free semigroup $\mathbb{F}^+_{\Lambda}$ inherited from $(\Lambda,\preccurlyeq)$.

\begin{definition}\label{def:order}
Suppose that $u=z_{\alpha_1}z_{\alpha_2}\cdots z_{\alpha_m}, v=z_{\beta_1}z_{\beta_2}\cdots z_{\beta_n}\in\mathbb{F}^+_{\Lambda}$, where $\alpha_i,\beta_j\in\Lambda$, $1\leqslant i\leqslant m$, $1\leqslant j\leqslant n$.
We say that $u\prec v$ if either of the following two conditions holds:
\begin{itemize}
\item[(1)] $m<n$;

\item[(2)] $m=n$, and there exists $1\leqslant i\leqslant n$ such that $\alpha_1=\beta_1,\cdots,\alpha_{i-1}=\beta_{i-1}$ and $\alpha_i\prec \beta_i$.
\end{itemize}
We use $u\preccurlyeq v$ to denote $u\prec v$ or $u=v$.
\end{definition}

It is apparent that $(\mathbb{F}^+_{\Lambda},\preccurlyeq)$ is a totally ordered set.
The order is multiplication-invariant in the following sense:
For any $w,u,v\in\mathbb{F}^+_{\Lambda}$, if $u\preccurlyeq v$, then
\begin{equation*}\label{eq:relation}
  wu\preccurlyeq wv,\quad uw \preccurlyeq vw,
\end{equation*}
and the equality holds if and only if $u=v$.
Moreover, $(\mathbb{F}^+_{\Lambda},\preccurlyeq)$ is well-ordered.

\begin{lemma}\label{lem:well-ordered set}
$(\mathbb{F}^+_{\Lambda}, \preccurlyeq)$ is a well-ordered set.
\end{lemma}

\begin{proof}
We assume that $\mathscr{S}$ is a non-empty subset of $\mathbb{F}^+_{\Lambda}$.
Let
\begin{equation*}
  n=\min\{|w|\colon w \in \mathscr{S}\},\quad \mathscr{S}_{0}=\{w\in \mathscr{S}\colon |w|=n\}.
\end{equation*}
Now let
\begin{equation*}
  \alpha_{1}=\min\{\alpha\in\Lambda\colon\text{there exists}~w\in\mathscr{S}_{0}~\text{such that}~z_{\alpha}|w\},
  \quad\mathscr{S}_{1}=\{w\in \mathscr{S}_{0}\colon z_{\alpha_1}|w\}.
\end{equation*}
Then we define
\begin{equation*}
  \alpha_{2}=\min\{\alpha\in\Lambda\colon\text{there exists}~ w \in \mathscr{S}_{1} ~\text{such that}~z_{\alpha_1}z_{\alpha} |w\},
  \quad \mathscr{S}_{2}=\{w\in\mathscr{S}_{1}\colon z_{\alpha_1}z_{\alpha_2}|w\}.
\end{equation*}
We proceed it by induction: for $2\leqslant k\leqslant n$,
\begin{align*}
  \alpha_{k}&=\min\{\alpha\in\Lambda\colon\text{there exists}~ w \in \mathscr{S}_{k-1} ~\text{such that}~z_{\alpha_1}\cdots z_{\alpha_{k-1}}z_{\alpha}|w\},\\
  \mathscr{S}_{k}&=\{w\in\mathscr{S}_{k-1}\colon z_{\alpha_1}\cdots z_{\alpha_{k-1}}z_{\alpha_k}|w\}.
\end{align*}
Then
\begin{equation*}
  \mathscr{S}_{n}=\{z_{\alpha_1}z_{\alpha_2}\cdots z_{\alpha_n}\}
\end{equation*}
consists of a single element since $\Lambda$ is a well-ordered set.
Thus $z_{\alpha_1}z_{\alpha_2} \cdots z_{\alpha_n}$ is the minimal element of $\mathscr{S}$.
\end{proof}

\begin{proposition}\label{prop:semisimple}
The algebras $\mathfrak{L}_{\Lambda}$ and $\mathfrak{A}_{\Lambda}$ are semisimple.
\end{proposition}

\begin{proof}
It follows from \cite[Lemma 4.4 and Lemma 4.5]{Hua21}.
\end{proof}

\begin{remark}
It was proved that $\mathfrak{L}_{\Lambda}$ is semisimple when $|\Lambda|=n$ in \cite{DP99}.
\end{remark}

\begin{theorem}\label{thm:continuity derivation}
All derivations of  $\mathfrak{L}_{\Lambda}$ and $\mathfrak{A}_{\Lambda}$ are automatically continuous.
\end{theorem}

\begin{proof}
By \cite[Theorem 4.1]{JS68}, every derivation on a semisimple Banach algebra is continuous.
Then we obtain the conclusion.
\end{proof}

In \cite{BC74}, Bade and Curtis studied the automatic continuity of derivations from a semisimple Banach algebra into a Banach bimodule and they provided several sufficient conditions.
We would like to ask the following question.

\begin{question}
Is each derivation $D\colon\mathfrak{A}_{\Lambda}\mapsto\mathfrak{L}_{\Lambda}$ automatically continuous?
\end{question}

\section{The first continuous cohomology groups}\label{sec:first cohomology}
A derivation $D$ of a unital Banach algebra $\mathcal{B}$ into a Banach  $\mathcal{B}$-bimodule $\mathcal{X}$ is called an {\em inner derivation} if there exists an element $T\in\mathcal{X}$ satisfying that $D(A)=AT-TA$ for all $A\in\mathcal{B}$.
We denote $D$ by $D_T$.
Let $\mathscr{D}(\mathcal{B},\mathcal{X})$ and $\mathscr{I}(\mathcal{B},\mathcal{X})$ be the linear space of all continuous derivations and inner derivations from $\mathcal{B}$ into $\mathcal{X}$, respectively.
The quotient space $\mathscr{H}^1(\mathcal{B},\mathcal{X})
=\mathscr{D}(\mathcal{B},\mathcal{X})/\mathscr{Inn}(\mathcal{B},\mathcal{X})$ is called {\em the first continuous cohomology group} of the Banach algebra $\mathcal{B}$ with coefficients in $\mathcal{X}$.

Recall that the non-commutative disc algebra $\mathfrak{A}_{\Lambda}$ is a Banach subalgebra of $\mathfrak{L}_{\Lambda}$.
We will compute the first continuous cohomology group $\mathscr{H}^1(\mathfrak{A}_{\Lambda},\mathfrak{L}_{\Lambda})$.
Specifically, we allow the number of generators to be uncountable.

\subsection{Locally inner derivations}
Let $D\in\mathscr{D}(\mathfrak{A}_{\Lambda},\mathfrak{L}_{\Lambda})$ be a continuous derivation.
We say that $D$ is {\em locally inner} at $w\in\mathbb{F}^+_{\Lambda}$ if there exists an operator $T_w\in\mathfrak{L}_{\Lambda}$ such that
\begin{equation*}
  D(L_w)=L_wT_w-T_wL_w.
\end{equation*}
In this subsection, we will show that every continuous derivation $D$ from $\mathfrak{A}_{\Lambda}$ into $\mathfrak{L}_{\Lambda}$ is locally inner at any $w\in\mathbb{F}^+_{\Lambda}$.

\begin{lemma}\label{lem:commutative zero}
Let $w\in\mathbb{F}^+_{\Lambda}$ and $D(L_w)=L_{\varphi}\in\mathfrak{L}_{\Lambda}$.
Then
\begin{equation*}
  \varphi(u)=0
\end{equation*}
for any $ u\in \mathbb{F}^+_{\Lambda}$ with $wu=uw$.
\end{lemma}

\begin{proof}
First we write
\begin{align*}
    D(L_w)=\sum\limits_{u\in \mathbb{F}^+_{\Lambda},\, wu=uw}\varphi(u)L_u+\sum\limits_{u\in \mathbb{F}^+_{\Lambda},\, wu\ne uw}\varphi(u)L_u.
\end{align*}
By induction, we have
\begin{equation*}\label{eq:D(g^k)}
  D(L_w^k)=L_w^{k-1}D(L_w)+L_w^{k-2}D(L_w)L_w+\cdots +D(L_w)L_w^{k-1}, \quad \forall k\geqslant 1.
\end{equation*}
It follows that
\begin{align*}
  D(L_w^k)&=kL_w^{k-1}\sum\limits_{u\in\mathbb{F}^+_{\Lambda},\,wu=uw} \varphi(u)L_u\\
  &\quad+\sum\limits_{u\in \mathbb{F}^+_{\Lambda},\,wu\ne uw}\varphi(u)(L_w^{k-1}L_u+L_w^{k-2}L_uL_w+\cdots+L_uL_w^{k-1}).
\end{align*}
Then
\begin{align*}
  \big\|D(L_w^k) \xi_{e}\big\|_2^2
  &=\left\|k\sum\limits_{u\in \mathbb{F}^+_{\Lambda},\, wu=uw}\varphi(u) \xi_{w^{k-1}u}\right\|_2^2\\
  &\quad+\left\|\sum\limits_{u\in \mathbb{F}^+_{\Lambda},\, wu\ne uw}\varphi(u)(\xi_{w^{k-1}u}+\xi_{w^{k-2}uw}+\cdots
  +\xi_{uw^{k-1}})\right\|_2^2.
\end{align*}
Since the derivation $D$ is continuous, for any $k\geqslant 1$, we have
\begin{equation*}
  k^2\sum\limits_{u\in \mathbb{F}^+_{\Lambda},\, wu=uw}|\varphi(u)|^2
  \leqslant\big\Vert D(L_w^k)\xi_{e}\big\Vert_2^2 \leqslant \Vert D\Vert^2
  <\infty.
\end{equation*}
Thus $\varphi(u)=0$ for any $u\in \mathbb{F}^+_{\Lambda}$ with $wu=uw$.
\end{proof}

\begin{lemma}\label{lem:small length zero}
Let $w\in\mathbb{F}^+_{\Lambda}$ and $D(L_w)=L_{\varphi}\in\mathfrak{L}_{\Lambda}$.
Then
\begin{equation*}
  \varphi(u)=0
\end{equation*}
for any $ u\in\mathbb{F}^+_{\Lambda}$ with $|u|<|w|$.
\end{lemma}

\begin{proof}
Let $w=z_{\alpha_1}z_{\alpha_2}\cdots z_{\alpha_n}$.
Then
\begin{equation*}
  D(L_w)=D(L_{\alpha_1})L_{\alpha_2}\cdots L_{\alpha_n}
  +L_{\alpha_1}D(L_{\alpha_2})L_{\alpha_3}\cdots L_{\alpha_n}+\cdots+L_{\alpha_1}\cdots L_{\alpha_{n-1}}D(L_{\alpha_n}).
\end{equation*}
Since the coefficient of $L_e$ in $D(L_{\alpha})$ vanishes for every $\alpha\in\Lambda$ by \Cref{lem:commutative zero}, it is easy to get the conclusion.
\end{proof}

\begin{lemma}\label{lem:conj and derivation}
Let $w\in\mathbb{F}^+_{\Lambda}$ and $D(L_w)=L_{\varphi}\in\mathfrak{L}_{\Lambda}$.
Then
\begin{equation*}
  D(L_w)=S_w-L_w^*S_wL_w
\end{equation*}
for some $S_w\in \mathfrak{L}_{\Lambda}$.
\end{lemma}

\begin{proof}
Since $D(L_e)=0$, we take $S_e=0$.
Now we assume that $w\ne e$.
By the definition of derivation and \Cref{lem:conjugation}, we have
\begin{equation*}\label{eq:g^{*(k-1)D(g^k)}}
  L_w^{*(k-1)}D(L_w^k)=\sum_{m=0}^{k-1}L_w^{*m}D(L_w)L_w^{m} \in \mathfrak{L}_{\Lambda}.
\end{equation*}
On the other hand, for any $u,v\in \mathbb{F}^+_{\Lambda}$, we have
\begin{equation*}
  \langle L_w^{*m}D(L_w)L_w^m \xi_{u}, \xi_{v}\rangle
  =\sum\limits_{h \in \mathbb{F}^+_{\Lambda}} \varphi(h)\langle \xi_{hw^{m}u}, \xi_{w^{m}v}\rangle=
  \begin{cases}
    \varphi(h), & hw^mu=w^mv,\\
    0, & \mbox{otherwise}.
  \end{cases}
\end{equation*}
If $hw^mu=w^mv$ for some $h\in \mathbb{F}^+_{\Lambda}$, then $|v|\geqslant|u|$ and there exists $u_{0} \in \mathbb{F}^+_{\Lambda}$ such that $v=u_{0}u$.
Hence $hw^m=w^mu_{0}$.
If $m\geqslant\frac{|u_{0}|}{|w|}+1$, then $wh=hw$ by \Cref{lem:commutative} and  $\varphi(h)=0$ by \Cref{lem:commutative zero}.
Therefore,
\begin{equation*}
\langle L_w^{*m}D(L_w)L_w^m \xi_{u}, \xi_{v}\rangle
=0, \quad \forall m \geqslant  \frac{|v|-|u|}{|w|}+1.
\end{equation*}
Since $\{L_w^{*(k-1)}D(L_w^k)\}_{k \geqslant 1}$ is a family of uniformly bounded operators in $\mathfrak{L}_{\Lambda}$, it converges to an operator $S_w\in\mathfrak{L}_{\Lambda}$ in the weak-operator topology, which is given by
\begin{equation*}
  \langle S_{w}\xi_{u},\xi_{v}\rangle
  =\lim\limits_{k\rightarrow \infty} \langle L_w^{*(k-1)}D(L_w^k)\xi_{u}, \xi_{v}\rangle,\quad\forall u, v \in \mathbb{F}^+_{\Lambda}.
\end{equation*}
Note that
\begin{equation*}
L_w^{*k}D(L_w^{k+1})
=L_w^{*k}\big(L_w^kD(L_w)+D(L_w^k)L_w\big)
=D(L_w)+L_w^* \big(L_w^{*(k-1)}D(L_w^k)\big)L_w.
\end{equation*}
Taking the weak-operator limit on the both sides, we obtain that
\begin{align*}
    S_{w}=D(L_w)+L_w^*S_{w}L_w.
\end{align*}
The proof is complete.
\end{proof}

Now we are able to show that $D$ is locally inner.

\begin{proposition}\label{prop:local inner derivation}
Let $w\in \mathbb{F}^+_{\Lambda}$, there exists $T_w\in \mathfrak{L}_{\Lambda}$ such that $D(L_w)=L_wT_w-T_wL_w$.
\end{proposition}

\begin{proof}
By \Cref{lem:conj and derivation}, there exists an operator $L_{\psi}\in \mathfrak{L}_{\Lambda}$ for some $\psi\in\mathcal{H}_{\Lambda}$ such that
\begin{equation*}
  D(L_w)=L_{\psi}-L_w^*L_{\psi}L_w.
\end{equation*}
Since $\Phi_j$ is contractive for any $j\geqslant0$, we can define
\begin{equation*}
  T_0:=L_{\psi}-\sum_{j=0}^{|w|-1}\Phi_{j}(L_{\psi})
  =\sum\limits_{u\in \mathbb{F}^+_{\Lambda},\,|u|\geqslant |w|}\psi(u)L_u\in \mathfrak{L}_{\Lambda}.
\end{equation*}
It is clear that
\begin{equation*}
  D(L_w)-(T_0-L_w^*T_0L_w)=\sum\limits_{u\in \mathbb{F}^+_{\Lambda},\,|u|<|w|}\psi(u)(L_u-L_w^*L_uL_w).
\end{equation*}
By \Cref{lem:small length zero}, both sides of the above equality are zero.
Hence
\begin{equation*}
  D(L_w)=T_0-L_w^*T_0L_w.
\end{equation*}
We write
\begin{equation*}
  T_0=L_w\sum\limits_{u\in \mathbb{F}^+_{\Lambda}}\psi(wu)L_u+\sum\limits_{u\in \mathbb{F}^+_{\Lambda},\,|u|\geqslant |w|,\, w\nmid u}\psi(u)L_u.
\end{equation*}
Since $L_w^*L_u=0$ when $|u|\geqslant |w|$ and $w\nmid u$, we can define
\begin{align*}
  T_w:=L_w^*T_0=\sum\limits_{u\in \mathbb{F}^+_{\Lambda}} \psi(wu)L_u\in\mathfrak{L}_{\Lambda}.
\end{align*}
Then
\begin{equation*}
  D(L_w)=T_0-L_w^*T_0L_w=L_wT_{w}-T_{w}L_w+\sum\limits_{u\in \mathbb{F}^+_{\Lambda},\,|u|\geqslant |w|,\, w\nmid u} \psi(u)L_u.
\end{equation*}
By induction, for any $k\geqslant 1$, we have
\begin{equation*}
  D(L_w^k)=L_w^kT_{w}-T_{w}L_w^k
  +\sum\limits_{u\in \mathbb{F}^+_{\Lambda},\,|u|\geqslant |w|,\, w\nmid u} \psi(u)(L_{w^{k-1}u}+L_{w^{k-2}uw}+\cdots+L_{uw^{k-1}}).
\end{equation*}
This implies that
\begin{equation*}
  k\sum\limits_{u\in \mathbb{F}^+_{\Lambda},\,|u|\geqslant |w|,\, w\nmid u}|\psi(u)|^2\leqslant\big\|(D(L_w^k)-L_w^kT_{w}+T_{w}L_w^k)\xi_{e}\big\|_2^2
  \leqslant(\|D\|+2\|T_w\|)^2<\infty.
\end{equation*}
Thus $\psi(u)=0$ for any $u\in\mathbb{F}^+_{\Lambda}$ with $|u|\geqslant |w|$ and $w\nmid u$.
It follows that
\begin{equation*}
  T_0=L_wT_w.
\end{equation*}
Hence $D(L_w)=L_wT_{w}-T_{w}L_w$.
\end{proof}

\subsection{Main result}
In this subsection, we prove that every continuous derivation $D$ from $\mathfrak{A}_{\Lambda}$ into $\mathfrak{L}_{\Lambda}$ is inner.

\begin{theorem}\label{thm:first cohomology}
The first continuous cohomology group of $\mathfrak{A}_{\Lambda}$ with coefficients in $\mathfrak{L}_{\Lambda}$ is zero, i.e., $H^1(\mathfrak{A}_{\Lambda},\mathfrak{L}_{\Lambda})=0$.
Equivalently, every continuous derivation of $\mathfrak{A}_{\Lambda}$ into $\mathfrak{L}_{\Lambda}$ is inner.
\end{theorem}

\begin{proof}
Let $\alpha\in \Lambda$.
By \Cref{prop:local inner derivation}, there exists $T_{\alpha}\in \mathfrak{L}_{\Lambda}$ such that
\begin{equation*}
  D(L_{\alpha})=L_{\alpha}T_{\alpha}-T_{\alpha}L_{\alpha}.
\end{equation*}
Considering the derivation $D-D_{T_{\alpha}}$, we may assume that $D(L_{\alpha})=0$.
Suppose that $|\Lambda|\geqslant 2$ and $\beta\in\Lambda\setminus\{\alpha\}$.
For any $n\geqslant 1$, it follows from \Cref{prop:local inner derivation} and $D(L_{\alpha})=0$ that
\begin{equation*}
  L_{\alpha}^nD(L_{\beta})L_{\alpha}^n=D(L_{z_{\alpha}^n z_{\beta}z_{\alpha}^n})
  =L_{z_{\alpha}^nz_{\beta}z_{\alpha}^n}T_{z_{\alpha}^nz_{\beta}z_{\alpha}^n}
  -T_{z_{\alpha}^nz_{\beta}z_{\alpha}^n}L_{z_{\alpha}^nz_{\beta}z_{\alpha}^n}
\end{equation*}
for some $T_{z_{\alpha}^nz_{\beta}z_{\alpha}^n}\in\mathfrak{L}_{\Lambda}$.
Let $T_{z_{\alpha}^nz_{\beta}z_{\alpha}^n}=L_{\varphi_n}$.
Then
\begin{equation*}
  L_{\alpha}^nD(L_{\beta})L_{\alpha}^n=
  \sum\limits_{w\in\mathbb{F}^+_{\Lambda}} \varphi_n(w)L_{z_{\alpha}^nz_{\beta}z_{\alpha}^nw}-\sum\limits_{w\in\mathbb{F}^+_{\Lambda}} \varphi_n(w)L_{wz_{\alpha}^nz_{\beta}z_{\alpha}^n}.
\end{equation*}
If $\varphi_{n}(w)\ne 0$, then we have
\begin{align*}
    z_{\alpha}^n|wz_{\alpha}^nz_{\beta}z_{\alpha}^n\quad \text{and}\quad z_{\alpha}^n\ddag z_{\alpha}^nz_{\beta}z_{\alpha}^nw,
\end{align*}
i.e.,
\begin{align*}
    wz_{\alpha}^nz_{\beta}z_{\alpha}^n=z_{\alpha}^n w_{1} \quad \text{and}\quad z_{\alpha}^nz_{\beta}z_{\alpha}^n w=w_{2}z_{\alpha}^n
\end{align*}
for some $w_{1}, w_{2} \in \mathbb{F}^+_{\Lambda}$.
It follows that that
\begin{align*}
  w=z_{\alpha}^k\quad\text{for some}~0\leqslant k\leqslant 2n-1
\end{align*}
or
\begin{align*}
  w=z_{\alpha}^nuz_{\alpha}^n\quad\text{for some}~u\in\mathbb{F}^+_{\Lambda}.
\end{align*}
Thus
\begin{equation*}
  D(L_{\beta})=\sum\limits_{u\in\mathbb{F}^+_{\Lambda}}
  \varphi_{n}(z_{\alpha}^nuz_{\alpha}^n)
  (L_{z_{\beta}z_{\alpha}^{2n}u}-L_{uz_{\alpha}^{2n}z_{\beta}})
  +\sum\limits_{k=1}^{2n-1}\varphi_{n}(z_{\alpha}^k)(L_{z_{\beta}z_{\alpha}^k}
  -L_{z_{\alpha}^kz_{\beta}}).
\end{equation*}
Let $D(L_{\beta})=L_{\varphi}$.
Then
\begin{equation*}
  \varphi(z_{\beta}z_{\alpha}^k)=-\varphi(z_{\alpha}^kz_{\beta})
  =\varphi_{n}(z_{\alpha}^k)\quad\text{for}~1\leqslant k\leqslant 2n-1,
\end{equation*}
and
\begin{align*}
    \varphi(w)=0\quad\text{for}~w\notin\{z_{\beta}z_{\alpha}^k, z_{\alpha}^kz_{\beta}\colon k\geqslant 0\}~\text{and}~|w|\leqslant 2n-1.
\end{align*}
By the arbitrariness of $n\geqslant 1$, we have
\begin{equation*}
  \varphi(z_{\beta}z_{\alpha}^k)=-\varphi(z_{\alpha}^kz_{\beta})\quad\text{for}~k\geqslant 1,
\end{equation*}
and
\begin{align*}
  \varphi(w)=0\quad\text{for}~w\notin\{z_{\beta}z_{\alpha}^k, z_{\alpha}^kz_{\beta}\colon k\geqslant 0\}.
\end{align*}
It follows that
\begin{equation}\label{eq:Derivation at b}
  D(L_{\beta})=\sum\limits_{k=1}^{\infty} \varphi(z_{\beta}z_{\alpha}^k)(L_{z_{\beta}z_{\alpha}^k}-L_{z_{\alpha}^kz_{\beta}}).
\end{equation}
Note that $L_{\beta}^*L_{\alpha}=0$.
Let
\begin{align*}
  T_{\beta}:=L_{\beta}^*D(L_{\beta})
  =\sum\limits_{k=1}^{\infty}\varphi(z_{\beta}z_{\alpha}^k)L_{z_{\alpha}^k} \in \mathfrak{L}_{\Lambda}.
\end{align*}
Then
\begin{align*}
  D(L_{\beta})=L_{\beta}T_{\beta}-T_{\beta}L_{\beta}.
\end{align*}
Note that $T_{\beta}L_{\alpha}=L_{\alpha}T_{\beta}$.
By considering $D-D_{T_{\beta}}$, we may assume that $D(L_{\alpha})=D(L_{\beta})=0$.
Suppose that $|\Lambda|\geqslant 3$ and $\gamma\in\Lambda\setminus\{\alpha,\beta \}$.
Let $D(L_{\gamma})=L_{\psi}$.
By \Cref{eq:Derivation at b} we can obtain
\begin{equation*}
  D(L_{\gamma})=\sum\limits_{k=1}^{\infty} \psi(z_{\gamma}z_{\alpha}^k)(L_{z_{\gamma}z_{\alpha}^k}-L_{z_{\alpha}^kz_{\gamma}})
  =\sum\limits_{k=1}^{\infty} \psi(z_{\gamma}z_{\beta}^k)(L_{z_{\gamma}z_{\beta}^k}-L_{z_{\beta}^kz_{\gamma}}).
\end{equation*}
It follows that $D(L_{\gamma})=0$.
Thus for any $w\in\mathbb{F}^+_{\Lambda}$,
\begin{align*}
    D(L_w)=0.
\end{align*}
Since the derivation $D$ is continuous, the desired result is obtained.
\end{proof}

\section{The first normal cohomology groups}\label{sec:normal cohomology}
A derivation of $\mathfrak{L}_{\Lambda}$ is called {\em normal} if it is weak-operator continuous on the closed unit ball $(\mathfrak{L}_{\Lambda})_1$.
Note that every inner derivations of $\mathfrak{L}_{\Lambda}$ is normal.
Let $\mathscr{D}_w(\mathfrak{L}_{\Lambda},\mathfrak{L}_{\Lambda})$ be the space of all normal derivations of $\mathfrak{L}_{\Lambda}$.
We define the {\em first normal cohomology group} of $\mathfrak{L}_{\Lambda}$ as the quotient space:
\begin{equation*}
  \mathscr{H}^1_w(\mathfrak{L}_{\Lambda},\mathfrak{L}_{\Lambda})
  =\frac{\mathscr{ D}_w(\mathfrak{L}_{\Lambda},\mathfrak{L}_{\Lambda})}
  {\mathscr{I}(\mathfrak{L}_{\Lambda},\mathfrak{L}_{\Lambda})}.
\end{equation*}

\begin{theorem}\label{thm:first normal cohomology}
The first normal cohomology group of $\mathfrak{L}_{\Lambda}$ is zero, i.e., $\mathscr{H}^1_w(\mathfrak{L}_{\Lambda},\mathfrak{L}_{\Lambda})=0$.
\end{theorem}

\begin{proof}
Let $D$ be a normal derivation on $\mathfrak{L}_{\Lambda}$.
Then the restriction of $D$ on $\mathfrak{A}_{\Lambda}$,
\begin{equation*}
  D|_{\mathfrak{A}_{\Lambda}}\colon \mathfrak{A}_{\Lambda}\to\mathfrak{L}_{\Lambda},
\end{equation*}
is a norm continuous derivation.
By \Cref{thm:first cohomology}, there exists an operator $T\in\mathfrak{L}_{\Lambda}$ such that
\begin{equation*}
  D(L_{\varphi})=L_{\varphi}T-TL_{\varphi},\quad\forall L_{\varphi}\in\mathfrak{A}_{\Lambda}.
\end{equation*}
Let $L_{\varphi}\in\mathfrak{L}_{\Lambda}$, $k\geqslant 1$ and $\mathcal{I}\subseteq\Lambda$ a finite set.
By \Cref{lem: Cesaro operators} and \Cref{prop: two algebras same}, we have the following three statements:
\begin{enumerate}[(i)]
\item $\mathcal{E}_{\mathcal{I}}(\Sigma_k(L_{\varphi}))$ is a finite sum; in particular, $\mathcal{E}_{\mathcal{I}}(\Sigma_k(L_{\varphi}))\in\mathfrak{A}_{\Lambda}$.

\item $\|\mathcal{E}_{\mathcal{I}}(\Sigma_k(L_{\varphi}))\|\leqslant \|\Sigma_k(L_{\varphi})\|$ and $\|\Sigma_k(L_{\varphi})\|\leqslant \|L_{\varphi}\|$.

\item $\lim_{\mathcal{I}}\mathcal{E}_{\mathcal{I}}(\Sigma_k(L_{\varphi}))
    =\Sigma_k(L_{\varphi})$ and $\lim_{k\to\infty}\Sigma_k(L_{\varphi})=L_{\varphi}$ both in the strong-operator topology.
\end{enumerate}
Therefore,
\begin{align*}
  D(\Sigma_k(L_{\varphi})) &=D(\lim_{\mathcal{I}}\mathcal{E}_{\mathcal{I}}(\Sigma_k(L_{\varphi})))
  =\lim_{\mathcal{I}}D(\mathcal{E}_{\mathcal{I}}(\Sigma_k(L_{\varphi})))\\
  &=\lim_{\mathcal{I}}D_T(\mathcal{E}_{\mathcal{I}}(\Sigma_k(L_{\varphi})))
  =D_T(\Sigma_k(L_{\varphi})),
\end{align*}
and
\begin{align*}
  D(L_{\varphi})=D(\lim_{k\to\infty}\Sigma_k(L_{\varphi}))
  =\lim_{k\to\infty}D(\Sigma_k(L_{\varphi}))
  =\lim_{k\to\infty}D_T(\Sigma_k(L_{\varphi}))=D_T(L_{\varphi}).
\end{align*}
We obtain that $D=D_T$ on $\mathfrak{L}_{\Lambda}$.
\end{proof}

In \Cref{thm:continuity derivation}, we have proved that all derivations of $\mathfrak{L}_{\Lambda}$ are continuous.
Kadison \cite{Kad66} proved that all derivations of a $C^*$-algebra are normal.
We propose the following question.
\begin{question}
Whether all derivations of $\mathfrak{L}_{\Lambda}$ are normal?
\end{question}

\section{The higher cohomology groups}\label{sec:higher order cohomology}
Let $\mathcal{B}$ be a unital Banach algebra over $\mathbb{C}$ and $\mathcal{X}$ a Banach $\mathcal{B}$-bimodule.
We denote the space of all $n$-linear continuous maps from the $n$-fold Cartesian product $\mathcal{B}^n$ into $\mathcal{X}$ by $\mathcal{L}^{n}(\mathcal{B},\mathcal{X})$ for $n\geqslant1$.
In this notation, $\mathcal{L}^{0}(\mathcal{B},\mathcal{X})$ is defined to be $\mathcal{X}$.
The coboundary operator $\partial^n\colon\mathcal{L}^{n}(\mathcal{B},\mathcal{X})
\to\mathcal{L}^{n+1}(\mathcal{B},\mathcal{X})$ is defined, for $n\geqslant1$, by
\begin{align*}
  \partial^n\phi(a_{1},a_{2},\ldots,a_{n+1})=&\ a_{1}\phi(a_{2},\ldots,a_{n+1})\\
  &+\sum_{i=1}^{n}(-1)^{i}
  \phi(a_1,\ldots,a_{i-1},a_ia_{i+1},a_{i+2},\ldots,a_{n+1})\\
  &+(-1)^{n+1}\phi(a_{1},\ldots,a_{n})a_{n+1},
\end{align*}
where $\phi\in\mathcal{L}^{n}(\mathcal{B},\mathcal{X})$ and $a_1,a_2,\ldots,a_{n+1}\in\mathcal{B}$.
When $n=0$, we define $\partial^0$ by
\begin{equation*}
  \partial^0 x(m)=mx-xm\quad (x\in\mathcal{X},\ m\in\mathcal{B}).
\end{equation*}
Since $\partial^{n}\partial^{n-1}=0$ for all $n\geqslant 1$, the image $\mathscr{I}(\partial^{n-1})$ is a linear subspace of
the kernel $\mathscr{K}(\partial^n)$.
An $n$-linear map $\phi\in\mathcal{L}^{n}(\mathcal{B},\mathcal{X})$ is called an $n$-cocycle if $\partial^n\phi=0$, and is called an $n$-coboundary if $\phi=\partial^{n-1}\psi$ for some $\psi\in\mathcal{L}^{n-1}(\mathcal{B},\mathcal{X})$.
 The {\em $n^{th}$ continuous Hochschild cohomology group} is defined as
\begin{align*}
  \mathscr{H}^{n}(\mathcal{B},\mathcal{X})
  =\frac{\mathscr{K}(\partial^n)}{\mathscr{I}(\partial^{n-1})}.
\end{align*}
 In particular, $\mathscr{H}^{1}(\mathcal{B},\mathcal{X})$ is the first continuous cohomology group studied in \S \ref{sec:first cohomology}.
We consider the complex field $\mathbb{C}$ as a Banach bimodule of the non-commutative disc algebra $\mathfrak{A}_{\Lambda}$ as follows:
For any $\gamma\in\mathbb{C}$ and $L_{\varphi}\in \mathfrak{A}_{\Lambda}$, we set
\begin{align*}
  \gamma\cdot L_{\varphi}=\gamma \varphi(e)\quad\text{and}\quad L_{\varphi}\cdot\gamma =\varphi(e)\gamma.
\end{align*}
On the other hand, for any $\alpha\in\Lambda$, we define
\begin{align*}
  \Upsilon_{\alpha}(L_{\varphi})=\sum_{z_{\alpha}|w}\varphi(w)L_w.
\end{align*}

\begin{lemma}\label{lem:Upsilon norm}
The linear operator $\Upsilon_{\alpha}$ is bounded with $\|\Upsilon_{\alpha}\|=2$.
\end{lemma}
\begin{proof}
Since
\begin{equation*}
  \Upsilon_{\alpha}(L_{\varphi})=L_{\alpha}L_{\alpha}^*(L_{\varphi}-\varphi(e)L_e)
  =L_{\alpha}L_{\alpha}^*(L_{\varphi}-\Phi_0(L_{\varphi})),
\end{equation*}
we have $\|\Upsilon_{\alpha}(L_{\varphi})\|\leqslant2\|L_{\varphi}\|$.
Since $\mathfrak{A}_{\Lambda}$ contains a copy of the classical disc algebra $A(\mathbb{D})$, we only need to prove the opposite inequality $\|\Upsilon_{\alpha}\|\geqslant 2$ for $A(\mathbb{D})$.
Consider the M\"{o}bius transformation
\begin{equation*}
  f(z)=\frac{c-z}{1-\overline{c}z}=\sum_{n=0}^{\infty}c_nz^n\in A(\mathbb{D}),\quad |c|<1.
\end{equation*}
Since $L_f$ is the Toeplitz operator of $f$, we have
\begin{equation*}
  \|L_{f}\|=\sup_{z\in\mathbb{D}}|f(z)|=1,\quad\text{and}\quad
  \|\Upsilon_{\alpha}(L_{f})\|=\sup_{z\in\mathbb{D}}|f(z)-f(0)|=1+|c|.
\end{equation*}
Then $\|\Upsilon_{\alpha}\|\geqslant 1+|c|$ for all $|c|<1$.
This completes the proof.
\end{proof}

Next we define the cutting operations on $\mathbb{F}_{\Lambda}^+$ by
\begin{align*}
  \chi(w)=\begin{cases}
  z_{\alpha_1},& w=z_{\alpha_1}z_{\alpha_2}\cdots z_{\alpha_n},n\geqslant 1,\\
  e,&\mbox{otherwise},
  \end{cases}
\end{align*}
and
\begin{align*}
  \Omega(w)=\begin{cases}
  z_{\alpha_2}\cdots z_{\alpha_n},&w=z_{\alpha_1}z_{\alpha_2}\cdots z_{\alpha_n},n\geqslant 1,\\
  e,&\mbox{otherwise}.
  \end{cases}
\end{align*}
It is clear that $w=\chi(w)\Omega(w)$ for all $w\in\mathbb{F}_{\Lambda}^+$.

\begin{theorem}\label{thm:higher order}
Suppose that $|\Lambda|<\infty$.
The higher cohomology groups of $\mathfrak{A}_{\Lambda}$ with coefficients in $\mathbb{C}$ are zero, i.e., $\mathscr{H}^n(\mathfrak{A}_{\Lambda},\mathbb{C})=0$ for any $n\geqslant 2$.
\end{theorem}

\begin{proof}
Let $\phi$ be an $n$-cocyle and $w_1,w_2,\ldots,w_n\in\mathbb{F}_{\Lambda}^+$.
If $w_1\ne e$, then
\begin{align*}
  0&=\partial^n\phi(L_{\chi(w_1)},L_{\Omega(w_1)},L_{w_2},\ldots,L_{w_n})\\
  &=L_{\chi(w_1)}\phi(L_{\Omega(w_1)},L_{w_2},\ldots,L_{w_n})\\
  &\quad-\phi(L_{w_1},L_{w_2},\ldots,L_{w_n})
  +\phi( L_{\chi(w_1)},L_{\Omega(w_1)w_2},L_{w_3},\ldots,L_{w_n})\\
  &\quad+\sum_{i=2}^{n-1}(-1)^{i+1}
  \phi(L_{\chi(w_1)},L_{\Omega(w_1)},L_{w_2},\ldots,L_{w_iw_{i+1}},\ldots,L_{w_{n}})\\
  &\quad+(-1)^{n+1}
  \phi(L_{\chi(w_1)},L_{\Omega(w_1)},L_{w_2},\ldots,L_{w_{n-1}})L_{w_n}\\
  &=-\phi(L_{w_1},L_{w_2},\ldots,L_{w_n})
  +\phi( L_{\chi(w_1)},L_{\Omega(w_1)w_2},L_{w_3},\ldots,L_{w_n})\\
  &\quad+\sum_{i=2}^{n-1}(-1)^{i+1}
  \phi(L_{\chi(w_1)},L_{\Omega(w_1)},L_{w_2},\ldots,L_{w_iw_{i+1}},\ldots,L_{w_{n}})\\
  &\quad+(-1)^{n+1}
  \phi(L_{\chi(w_1)},L_{\Omega(w_1)},L_{w_2},\ldots,L_{w_{n-1}})L_{w_n}.
\end{align*}
It follows that
\begin{align*}
  \phi(L_{w_1},L_{w_2},\ldots,L_{w_n})
  &=\phi( L_{\chi(w_1)},L_{\Omega(w_1)w_2},L_{w_3},\ldots,L_{w_n})\\
  &\quad+\sum_{i=2}^{n-1}(-1)^{i+1}
  \phi(L_{\chi(w_1)},L_{\Omega(w_1)},L_{w_2},\ldots,L_{w_iw_{i+1}},\ldots,L_{w_{n}})\\
  &\quad+(-1)^{n+1}
  \phi(L_{\chi(w_1)},L_{\Omega(w_1)},L_{w_2},\ldots,L_{w_{n-1}})L_{w_n}.
\end{align*}
Similarly, we have
\begin{align*}
  \phi(L_e,L_{w_2},\ldots,L_{w_n})
  &=-\sum_{i=2}^{n-1}(-1)^{i+1}
  \phi(L_e,L_e,L_{w_2},\ldots,L_{w_iw_{i+1}},\ldots,L_{w_{n}})\\
  &\quad-(-1)^{n+1}\phi(L_e,L_e,L_{w_2},\ldots,L_{w_{n-1}})L_{w_n}.
\end{align*}
Suppose that $\psi$ is an $(n-1)$-linear map such that
\begin{equation*}
  \psi(L_{w_1},L_{w_2},\ldots,L_{w_{n-1}})=
  \begin{cases}
    -\phi(L_{\chi(w_1)},L_{\Omega(w_1)},L_{w_2},\ldots,L_{w_{n-1}}), & w_1\ne e, \\
    \phi(L_e,L_e,L_{w_2},\ldots,L_{w_{n-1}}), & w_1=e.
  \end{cases}
\end{equation*}
By a direct computation, we have
\begin{equation*}
  \partial^{n-1}\psi(L_{w_1},\ldots,L_{w_n})=\phi(L_{w_1},\ldots,L_{w_n}).
\end{equation*}
On the other hand, the $(n-1)$-linear map $\psi$ is given by
\begin{align*}
  &\quad\psi(L_{\varphi_1},L_{\varphi_2},\ldots,L_{\varphi_{n-1}}) \\
  &=-\sum_{\alpha\in\Lambda}
  \phi(L_{\alpha},L_{\alpha}^*\Upsilon_{\alpha}(L_{\varphi_1}),
  L_{\varphi_2},\ldots,L_{\varphi_{n-1}})
  +\varphi_1(e)\phi(L_e,L_e,L_{\varphi_2},\ldots,L_{\varphi_{n-1}}).
\end{align*}
It follows from \Cref{lem:Upsilon norm} that
\begin{equation*}
  \|\psi(L_{\varphi_1},L_{\varphi_2},\ldots,L_{\varphi_{n-1}})\|\leqslant (2|\Lambda|+1)\|\phi\|\prod_{i=1}^{n-1}\|L_{\varphi_i}\|,
\end{equation*}
and hence $\psi$ is continuous.
Therefore, $\partial^{n-1}\psi=\phi$.
\end{proof}

\begin{remark}
For the first cohomology group, Popescu \cite[Theorem 4.1]{Pop96} proved that $\mathscr{H}^1(\mathfrak{A}_{\Lambda},\mathbb{C})\cong (\mathbb{C}^{|\Lambda|},+)$ when $|\Lambda|<\infty$ and $\mathscr{H}^1(\mathfrak{A}_{\Lambda},\mathbb{C})\cong (\ell^2(\Lambda),+)$ when $|\Lambda|=\infty$.
\end{remark}

We end this paper by proposing the following question.
\begin{question}
When we choose a different Banach bimodule $\mathcal{X}$, with $\mathcal{X}=\mathfrak{L}_{\Lambda}$ as an important example, what are the higher cohomology groups $\mathscr{H}^n(\mathfrak{A}_{\Lambda},\mathcal{X}), n\geqslant 2$?
\end{question}

\end{document}